\documentclass[leqno,12pt,draft]{amsart}

\usepackage{amssymb}
\usepackage{amsmath}
\usepackage{enumerate}
\usepackage{amsfonts}
\usepackage{color}

\numberwithin{equation}{section}

\newtheorem{lema}[equation]{Lemma}
\newtheorem{teo}[equation]{Theorem}

\newtheorem{coro}[equation]{Corollary}
\begin{document}
\title[Riesz transform for Schr\"odinger operator]{Potential-free  $L^1$-estimates for positivity-preserving Riesz transform related to Schr\"odinger operator\\ in dimension one }
\author{Jacek Dziuba\'{n}ski }
\address{Instytut Matematyczny, Uniwersytet Wroc\l awski, pl. Grunwaldzki 2/4, 50-384 Wroc\l aw, Poland} \email{jdziuban@math.uni.wroc.pl }
\thanks{}

\maketitle

{\it In memory of Jacek Zienkiewicz (1967--2023)}

\begin{abstract} Let $V\geq 0$ be a locally integrable function on $\mathbb R$. Consider the Schr\"odinger operator $L=-\frac{d^2}{dx^2} +V$. We prove that for all $0<a\leq 1$, there is a constant $C_a$, independent of $V$, such that the Riesz transform type operator $V^aL^{-a}$ is bounded on $L^1(\mathbb R)$ and $\| V^aL^{-a}f\|_{L^1(\mathbb R)}\leq C_a\|f\|_{L^1(\mathbb R)}$.

\end{abstract}


\section{Introduction} In \cite{K-W} Kucharski and Wróbel studied positivity-preserving Riesz transforms
$$R^a_V=V^{a}(-\Delta+V)^{-a}$$ associated with the  Schr\"odinger operator $L=-\Delta+V$, where $V\in L^1_{\rm loc}(\mathbb R^d)$ is a non-negative potential and $0\leq a\leq 1$. They showed that for any fixed $p\in (1,2]$ and $0\leq a\leq 1/p$ there is a constant $C=C(a,p)$, independent of $V$ and  $d$, such that
\begin{equation}\label{KW}
 \big\|V^aL^{-a}f\big\|_{L^p(\mathbb R^d)} \leq C\|f\|_{L^p(\mathbb R^d)}.
\end{equation}

Furthermore,   if $V(\boldsymbol x)=V_1(x_1)+V_2(x_2)+\dots + V_d(x_d)$, where $\boldsymbol x=(x_1,x_2,\dots, x_d)$,
$m|x_j|^\alpha\leq V_j(x_j)\leq M|x_j|^\alpha$ with some $0<\alpha\leq 2$, and $0<m\leq M$,  then for all $0<a\leq 1$, the operator  $V^aL^{-a}$ is bounded on $L^1(\mathbb R^d)$ (see  Kucharski \cite[Theorem 1.2]{K2}).

The goal of this note is to prove that in dimension one the estimates \eqref{KW} hold for $p=1$ and any locally integrable non-negative potential $V$, and are independent of $V$, which is stated in the following theorem.
\begin{teo}\label{main_theo}
    Let $0<a\leq 1$. Then there is a constant $C_a>0$ such that for any $V\in L^1_{\rm loc}(\mathbb R)$, $V\geq 0$,  one has

    \begin{equation}
        \Big\|V^a\Big(-\frac{d^2}{dx^2}+V\Big)^{-a} f\Big\|_{L^1(\mathbb R)}\leq C_a\| f\|_{L^1(\mathbb R)}.
    \end{equation}
\end{teo}

\begin{coro}\label{Coro1}
    Suppose $V(\boldsymbol x)=V_1(x_1)+V_2(x_2)+...+V_d(x_d)$, where $V_j\in L^1_{\rm loc} (\mathbb R)$, $V_j\geq 0$, $j=1,2,\dots,d$, $\boldsymbol x=(x_1,x_2,\dots , x_d)\in \mathbb R^d$. Then for $0<a\leq 1$,
    $$\| V^a(-\Delta +V)^{-a}f\|_{L^1(\mathbb R^d)}\leq C_ad\|f\|_{L^1(\mathbb R^d)}.$$
\end{coro}

\begin{coro}\label{Coro2}
    Suppose $V\in L^1_{\rm loc}(\mathbb R^d)$ is such that there are non-negative functions $V_1,V_2,\dots, V_d\in L^1_{\rm loc}(\mathbb R)$ and constants $C,c>0$ such that
    \begin{equation*}
        c(V_1(x_1)+V_2(x_2)+\dots V_d(x_d))\leq V(x_1,x_2,\dots, x_d)\leq C(V_1(x_1)+V_2(x_2)+\dots V_d(x_d)).
    \end{equation*}
    Then, for $0<a\leq 1$,
    \begin{equation}\label{eq:VV}
        \|V^a(-\Delta+V)^{-a} f\|_{L^1(\mathbb R^d)}\leq \Big(\frac{C}{c}\Big)^aC_a d\|f\|_{L^1(\mathbb R^d)}.
    \end{equation}
\end{coro}
The proofs of Theorem \ref{main_theo} and the corollaries are contained in Section \ref{Section4}.
\section{Schr\"odinger operators and associated Riesz transforms}
For $V\in L^1_{\rm loc}(\mathbb R)$, $V\geq 0$, consider the quadratic form
\begin{equation}\label{form} Q(f,f)=\int_{\mathbb R} \frac{d}{dx}f(x)\overline{\frac{d} {dx}f(x)}\, dx +\int_{\mathbb R} V(x)f(x)\overline{f(x)}\, dx,
\end{equation}
with its domain $\mathcal D(Q)=\{f\in W_2^1:\int_{\mathbb R}V(x)|f(x)|^2<\infty\}$, where
$$W_2^1(\mathbb R)=\{f\in L^2(\mathbb R): \xi \widehat f(\xi)\in L^2(\mathbb R)\}$$ is the inhomogeneous Sobolev space. Here, $\hat f(\xi)$ is the Fourier transform of $f$.  The form $Q$ on its domain is closed and defines a non-negative self-adjoint Schr\"odinger operator $L$ denoted for simplicity by $L=-\frac{d^2}{dx^2}+V$.
Let
$$Lf=\int_0^\infty \lambda\, dE(\lambda)f$$
be its spectral resolution.

The quadratic form \eqref{form} can also be read as
\begin{equation}
    \label{form2}
    \Big\|\frac{d}{dx} f\Big\|_{L^2(\mathbb R)}^2 + \big\|V^{1/2}f\big\|_{L^2(\mathbb R)}^2= \big\|L^{1/2} f\big\|_{L^2(\mathbb R)}^2
\end{equation}

For $\varepsilon>0$, let $L_\varepsilon=(\varepsilon I+L)$. Then $L_\varepsilon^{-1/2}$ is a bounded operator on $L^2(\mathbb R)$ and, by \eqref{form2},
\begin{equation}\label{form3}
    \Big\| \frac{d}{dx} L_\varepsilon^{-1/2} f\Big\|_{L^2(\mathbb R)}^2+ \big\|V^{1/2}L_\varepsilon^{-1/2}\big\|_{L^2(\mathbb R)}^2=\big\| L^{1/2} L_\varepsilon^{-1/2}f\big\|_{L^2(\mathbb R)}^2, \quad f\in L^2(\mathbb R).
\end{equation}
It follows from \eqref{form2} that 0 is not an eigenvalue for $L$, that  is, $E(\{0\})=0$. Hence, by the spectral theorem, the operators $L^{1/2}L_\varepsilon^{-1/2}$ converge strongly to the identity operator, as $\varepsilon$ tents to 0. Consequently, by \eqref{form2}, for $f\in L^2(\mathbb R)$ the limits
$$\lim_{\varepsilon \to 0} \frac{d}{dx} R_\varepsilon^{-1/2} f\quad \text{\rm and} \quad  \lim_{\varepsilon \to 0} V^{1/2} R_\varepsilon^{-1/2} f$$
exist in the $L^2(\mathbb R)$-norm and define bounded linear operators, denoted $R^{1/2}$ and $R_V^{1/2}$, respectively. Moreover,
$$ \big\|R^{1/2}f\big\|_{L^2(\mathbb R)}^2+\big\| R_V^{1/2}f\big\|_{L^2(\mathbb R)}^2 =\| f\|_{L^2(\mathbb R)}^2.$$

Let
$$Q_V(f,f)=\int_{\mathbb R} V(x)f(x)\overline{f(x)}\, dx$$
be the quadratic form of the non-negative linear densely defined self-adjoint operator $f\mapsto V(x)f(x)$. Since $0\leq Q_V(f,f)\leq Q(f,f)$, for $f\in\mathcal D(Q)$,
\begin{equation}
 \|V^{\beta/2}f\|_{L^2(\mathbb R)}^2=   \int_{\mathbb R} V^{\beta}(x) f(x)\overline{f(x)}\, dx \leq \int_{\mathbb R} L^{\beta/2}f(x)\overline{L^{\beta/2}f(x)}\, dx =\|L^{\beta/2}f\|_{L^2(\mathbb R)}^2
\end{equation}
for all $0<\beta\leq 1$ and all $f\in \mathcal D(L^{\beta/2})=\{f\in L^2(\mathbb R): \int_0^\infty \lambda^{\beta}\, dE_{f,f}(\lambda)<\infty\}$ (see e.g. Davies \cite[Lemma 4.20]{Davies}) .

For $0<\beta\leq 1$ and $\varepsilon >0$,  the operators $\mathcal R^{\beta/2}_{\varepsilon , V} =V^{\beta/2}(\varepsilon I+L)^{-\beta/2}$ are bounded on $L^2(\mathbb R)$.   Similarly to  above, the limit
$$\lim_{\varepsilon\to 0} V^{\beta/2}(\varepsilon I+L)^{-\beta/2}f$$ exists for all $f\in L^2(\mathbb R)$  and defines a contraction on $L^2(\mathbb R)$, denoted by $V^{\beta/2}L^{-\beta/2}$. Our result stated in Theorem \ref{main_theo} asserts in particular that the operator, initially defined on $L^2(\mathbb R)\cap L^1(\mathbb R)$, has a unique extension to a bounded operator on $L^1(\mathbb R)$ and its norm is bounded by a~universal constant $C_{\beta/2}$  (independent of $V$). Actually, the bound $L^1$ for $V^{\beta/2}L^{-\beta/2}$ holds for a wider range of parameters, namely for $0<\beta\leq 2$.

\section{Kernels of operators}

The operator $-L$ generates a strongly continuous semigroup $\{K_t\}_{t>0}$ of self-adjont contractions on $L^2(\mathbb R)$. By the Feynman-Kac formula the semigroup $\{K_t\}_{t>0}$ has the form
\begin{equation}\label{integral} K_tf(x)=\int_{\mathbb R} k_t(x,y)f(y)\, dy,
\end{equation}
where
\begin{equation}\label{k<h}
    0<k_t(x,y)\leq \frac{1}{\sqrt{4\pi t}} e^{-|x-y|^2/4t}.
\end{equation}
If $1\leq p<\infty$, then formula \eqref{integral} defines a strongly continuous semigroup of linear contractions on $L^p(\mathbb R)$.

The negative power $(\varepsilon I+L)^{-a}$  of the operator  $(\varepsilon I+K)$, which is a bounded operator on $L^p(\mathbb R)$,   is given by
\begin{equation}
    (\varepsilon I+L)^{-a}f=\Gamma (a)^{-1} \int_0^\infty t^ae^{-\varepsilon t} K_t f\frac{dt}{t}.
\end{equation}
Consequently,
$$ R_{\varepsilon, V}^a(x,y)=\Gamma (a)^{-1}V^a(x)\int_0^\infty t^{a}e^{-\varepsilon t}k_t(x,y)\frac{dt}{t}$$
is the integral kernel of the operator $V^{a}(\varepsilon I+L)^{-a}$. Observe that
$$0\leq  R_{\varepsilon, V}^a(x,y)\leq R_{\varepsilon', V}^a(x,y)\quad \text{\rm for } \ 0<\varepsilon'\leq \varepsilon $$
and
\begin{equation}\label{kernel-R_a}\lim_{\varepsilon \to 0^+} R_{\varepsilon, V}^a(x,y)=\Gamma(a)^{-1}V(x)^a\int_0^\infty t^ak_t(x,y)\frac{dt}{t}=: R^a_V(x,y).
\end{equation}
Thus, the proof of Theorem \ref{main_theo} reduces to proving that for $0<a\leq 1$ there is a constant $C_a$ independent of $V$ such that
\begin{equation}\label{enough} \sup_{y\in\mathbb R} \int_{\mathbb R}  R^{a}_{V}(x,y)\, dx =\sup_{y\in\mathbb R} \Gamma(a)^{-1}\int_{\mathbb R}\int_0^\infty V(x)^at^a k_t(x,y)\frac {dt}{t}\, dx\leq C_a.
\end{equation}

\section{Estimates of kernels - proof of Theorem \ref{main_theo} and corollaries}\label{Section4}

The proof of Theorem \ref{main_theo} uses methods  of Czaja-Zienkiewicz \cite{CZ}. For the convenience of the reader, we provide all the details.

\subsection{Cover of $\mathbb R$}
For an interval $I$ and a nonnegative number $a$,  we denote by $aI$ the interval with the same center $x_I$ of $I$ and the length $a$-times longer.
If $I$ is a dyadic interval, then $\tilde I$ denotes a unique dyadic interval such that $I\subset \tilde I$, and $|\tilde I|=2|I|$, where $|I|$ stands for the length of $I$.

Fix a non-negative locally integrable potential $V$ on $\mathbb R$. Following \cite{CZ}, we say that an interval $I$ belongs to $ \mathcal I$, if $I$ is a (unique) maximal dyadic interval satisfying
\begin{equation}\label{def-I} |I|\int_{16I} V(y)\, dy \leq 1.
\end{equation}

\begin{lema}
    Assume  that $I,J\in\mathcal I$, $|I|\leq |J|$, and $\text{\rm dist}\,(I,J)\leq |J|$. Then
    $$ |I|\geq \frac{1}{4} |J|.$$
\end{lema}
\begin{proof}
     Aiming for a contradiction suppose that $|I|<\frac{1}{4}|J|$, which means that $|I|\leq\frac{1}{8}|J|$. Then $16\tilde I\subset 16J$. Hence,
     $$ |\tilde I| \int_{16\tilde I}V(y)\, dy \leq \frac{1}{4} |J|\int_{16J} V(y)\, dy \leq \frac{1}{4},$$
     where in the last inequality, we have used \eqref{def-I} for $J$ instead of $I$.
   On the other hand,  by  the definition of $I$,
    $$1< |\tilde I|\int_{16\tilde I}V(y)\, dy $$
   and we arrive at a contradiction.
\end{proof}

Observe that if $V\not\equiv 0$, then the intervals  $I$ from the set $\mathcal I$ have disjoint interiors and form  a cover of $\mathbb R$.

\subsection{Subharmonic functions}\label{subharmonic}
By continuity, for each $I\in\mathcal I$, there is an interval $I^{\diamond}$ such that  $I\subseteq I^\diamond \subset \tilde I$ and
\begin{equation}\label{diamond}
    |I^\diamond |\int_{16I^\diamond} V(y)\, dy=1.
\end{equation}
We define
\begin{equation}\label{def-psi} \psi_I(x)=1+ \int_{16I^{\diamond}}V(y)\frac{1}{2}|x-y|\, dy.
\end{equation}
\begin{lema}
    \begin{equation}\label{psi_in} \frac{1}{64} \Big(1+\frac{|x-x_I|}{|I|}\Big) \leq \psi_I(x)\leq 64 \Big(1+\frac{|x-x_I|}{|I|}\Big).
    \end{equation}
\end{lema}
\begin{proof}
    If $x\in 64 I$ and $y\in 16I^{\diamond}$,   then $|x-y|\leq 64 |I|$. Hence, using the definition of $\psi_I$ together with \eqref{diamond}, we get
    $$1\leq \psi_I(x)\leq1+ 32\int_{16 I^{\diamond}}V(y)|I|\, dy \leq 33,$$
    which gives \eqref{psi_in} in this case.

       Now, if $x\notin 64 I$ and $y\in 16 I^{\diamond}$, then $|x-x_I|/4\leq |x-y|<|x-x_I|+17|I|$, and, consequently,  \eqref{psi_in} follows by applying \eqref{def-psi} together with \eqref{diamond}.
\end{proof}

\begin{lema}
    Let $I\in\mathcal I$. For $y\in  I$, and $t>|I|^2$, we have
    \begin{equation}\label{zenek}
        \int_{\mathbb R} k_t(x,y)\, dx \leq (1+2\cdot 64^2)\frac{|I|^{1/2}}{t^{1/4}}
    \end{equation}
\end{lema}
\begin{proof}
    The lemma is included in the proof of Lemma 2.2 of \cite{CZ}.
    The only task is to keep an eye on the constants. For the reader's convenience, we repeat the arguments after \cite{CZ}.

 Since $\frac{1}{2}|x|$ is the fundamental solution for $\frac{d^2}{dx^2}$, we have
$$ \psi_I''(x)=V(x)\chi_{I^{\diamond}}(x),$$
and, consequently, $\psi_I$ is a subharmonic function of $L$, that is,
\begin{equation}\label{negative} -L\psi_I\leq V\chi_{I^{\diamond}}-V\leq 0.\end{equation}
Since  $K_t$ preserve the positivity of functions, inequity \eqref{negative} implies
\begin{equation}\label{monotone}
    K_t\psi_I-\psi_I=-\int_0^tK_sL\psi_I\, ds\leq 0.
\end{equation}
 For $x\in I$, using \eqref{psi_in} together with  \eqref{monotone},  we obtain
\begin{equation}\begin{split}\label{eq:64}
    \int_{\mathbb R} k_t(x,y)\Big(1+\frac{|y-x_I|}{|I|}\Big)\, dy &\leq 64K_t\psi_I(x)\leq 64 \psi_I(x)\\
    & \leq 64^2\Big(1+\frac{|x-x_I|}{|I|}\Big) \leq 2 \cdot  64^2.
\end{split}\end{equation}
So, using \eqref{k<h} together with \eqref{eq:64}, for $x\in I$ and $R>0$, we have
\begin{equation*}
    \begin{split}
        \int k_t(x,y)\, dy&= \int_{|y-x_I|<R}  k_t(x,y)\, dy+\int_{|y-x_I|\geq R}  k_t(x,y) \Big(1+\frac{|y-x_I|}{|I|}\Big) \Big(1+\frac{R}{|I|}\Big)^{-1}\, dy\\
        &\leq \frac{1}{\sqrt{4\pi t}} R+2\cdot 64^2\frac{|I|}{R}.
    \end{split}
\end{equation*}
Setting $R=\sqrt{|I|}t^{1/4}$, we get
\begin{equation}\label{int_k_t}
    \begin{split}
        \int k_t(x,y)\, dy&\leq (1+ 2\cdot 64^2)\frac{|I|^{1/2}}{t^{1/4}}\quad \text{for } x\in I.
    \end{split}
\end{equation}
Now we may swap $x$ and  $y$ in \eqref{int_k_t}, because $k_t(x,y)=k_t(y,x)$.
\end{proof}

\begin{lema}\label{Lemma_a=1}
    For all $y\in\mathbb R$, one has
    \begin{equation}\label{bound_1}
      \int_0^\infty   \int_{\mathbb R} V(x)k_t(x,y)\, dx\, dt\leq 1.
    \end{equation}
\end{lema}
\begin{proof}
    The lemma is well-known and follows from the perturbation formula
    \begin{equation}\label{perturb}
        h_t(x,y)=k_t(x,y)+\int_0^t\int_{\mathbb R} h_{t-s}(x,z)V(z)k_s(z,y)\, dz\, ds.
    \end{equation}
 Indeed, integrating the equation with respect to $dx$, and using \eqref{k<h}, we get
    $$1\geq \int_0^t \int_{\mathbb R} V(z)k_s(z,y)\, dz\, ds.$$
    Letting $t\to\infty$, we get the desired bound.
\end{proof}

\subsection{Proof of Theorem \ref{main_theo}}

\begin{proof}[Proof of Theorem \ref{main_theo}] The theorem for $a=1$ is Lemma \ref{Lemma_a=1}.  In order to prove the theorem for $0<a<1$, it suffices to establish \eqref{enough}, where the kernel $R_{V}^a(x,y)$ is defined in \eqref{kernel-R_a}.

Fix $y\in\mathbb R$. Let $I\in\mathcal I$ be such that $y\in I$.
We split  the integral \eqref{enough} into three parts.

{\bf Part 1.} By the H\"older inequality, with $p=1/a$ and $p'=1/(1-a)$, we get
\begin{equation}
    \begin{split}
     J_1:=  &\Gamma(a)^{-1} \int_{\mathbb R} \int_{|I|^2}^\infty  V^a(x)k_t(x,y)\, t^{a-1}\, dx \, dt \\
       &\leq \Gamma(a)^{-1}\Big(\int_{\mathbb R}\int_{|I|^2}^\infty V(x)k_t(x,y)dt\, dx\Big)^{a}\Big(\int_{\mathbb R}\int_{|I|^2}^\infty k_t(x,y)\frac{dt}{t}\, dx\Big)^{1-a}\\
       \end{split}
       \end{equation}
Applying \eqref{bound_1} to the first factor and \eqref{zenek} to the second, we obtain

       \begin{equation}
           \begin{split}
     J_1  &\leq \Gamma(a)^{-1}\Big(\int_{|I|^2}^\infty (1+2\cdot 64^2)\frac{|I|^{1/2}}{t^{1/4}}\frac{dt}{t}\, dx\Big)^{1-a}\leq \Gamma(a)^{-1}(1+2\cdot 64^2)^{1-a}4^{1-a}.
    \end{split}
\end{equation}
{\bf Part 2.} Using the Fubini theorem, the H\"older inequality, \eqref{k<h}, and the definition
of $I^{\diamond}$ (see Subsection \ref{subharmonic}), we get

\begin{equation}
    \begin{split}
     J_2:=  & \Gamma(a)^{-1}\int_{16I^{\diamond}} \int_0^{|I|^2} V^a(x)k_t(x,y)\, t^{a-1}\, \, dt\, dx\\
       & \leq \Gamma(a)^{-1} \int_0^{|I|^2}\Big(\int_{16I^{\diamond}} V(x)k_t(x,y) \, dx\Big)^{a}\Big(\int_{16 I^{\diamond}}k_t(x,y)\, dx\Big)^{1-a}t^{a-1}\, dt\\
       &\leq \Gamma(a)^{-1}\int_0^{|I|^2} (|I|^{-1} (4\pi t)^{-1/2})^a t^{a-1}\, dt
       \leq \Gamma(a)^{-1}\frac{2}{a} (4\pi)^{-a/2}.
    \end{split}
\end{equation}
{\bf Part 3.} Again, by the Fubini theorem,  the H\"older inequality together with \eqref{k<h} and \eqref{bound_1}, we obtain
\begin{equation*}
    \begin{split}
      J_3: & =\Gamma(a)^{-1} \int_{(16I^{\diamond})^c} \int_0^{|I|^2}V^a(x)k_t(x,y)\, t^{a-1}\,  dt\, dx \\
       &\leq \Gamma(a)^{-1}\Big(\int_0^{|I|^2}\int_{(16I^{\diamond})^c} V(x)k_t(x,y)\, dx\, dt\Big)^a \Big(\int_0^{|I|^2}\int_{(16I^{\diamond})^c} \frac{1}{\sqrt{4\pi t}} e^{|x-y|^2/4t}\, dx\, \frac{dt}{t}\Big)^{1-a} \\
       &\leq \Gamma(a)^{-1}\Big(\int_0^{|I|^2}\int_{(16I^{\diamond})^c} \frac{1}{\sqrt{4\pi t}} e^{|x-x_I|^2/16t}\, dx\, \frac{dt}{t}\Big)^{1-a},\\
    \end{split}
\end{equation*}
where in the last inequality we have used the relation $y\in I$.
Finally, observe that
\begin{equation}
    \begin{split}
        \int_0^{|I|^2}\int_{(16I^{\diamond})^c} \frac{1}{\sqrt{4\pi t}} e^{|x-x_I|^2/16t}\, dx\, \frac{dt}{t}&
        \leq \int_0^{|I|^2} \int_{|x|>|I|} \frac{1}{\sqrt{4\pi t}} e^{-|x|^2/16t}  \, dx \, \frac{dt}{t}\\
        &=\frac{1}{\sqrt{4\pi}}\int_{|x|>1} 2 \ln |x|e^{-x^2/16}\, dx=: C'.
    \end{split}
\end{equation}
Hence $J_3\leq \Gamma(a)^{-1} C'^{(1-a)}$.

Combining Parts 1-3, we obtain the bound \eqref{enough} with $C_a$ independent of $V$.
\end{proof}

\subsection{Proofs of Corollaries \ref{Coro1} and \ref{Coro2}}

\begin{proof}[Proof of Corollary \ref{Coro1}]
    Let $k_t^{\{j\}}(x,y)$ denote the integral kernel of the semigroup generated by the one dimensional Schr\"odinger operator $\frac{d^2}{dx^2}-V_j(x)$. Then the integral kernel $k_t(\boldsymbol x,\boldsymbol y)$ of the semigroup generated by $\Delta-V$ has the form
    \begin{equation*}
        k_t(\boldsymbol x,\boldsymbol y)=\prod_{j=1}^d k^{\{j\}}_t(x_j,y_j).
        \end{equation*}
        So, the corresponding integral kernel $R_V^a(\boldsymbol x,\boldsymbol y)$ of the Riesz transform $V^a(-\Delta+V)^{-a}$ is estimated as follows:
        \begin{equation}\label{R-bound2}
            \begin{split}
              0\leq  R_V^a(\boldsymbol x,\boldsymbol y)&=\Gamma(a)^{-1}\Big(\sum_{j=1}^dV_j(x_j)\Big)^a\int_0^\infty t^a\prod_{j=1}^d k^{\{j\}}_t(x_j,y_j)\frac{dt}{t}\\
                &\leq \sum_{j=1}^d \Gamma(a)^{-1} V_j(x_j)^a \int_0^\infty t^a\prod_{j=1}^d k_t^{\{j\}}(x_j,y_j)\frac{dt}{t}.
            \end{split}
        \end{equation}
        Using \eqref{R-bound2}, \eqref{k<h}, Theorem \ref{main_theo} together with the Fubini theorem, we conclude that
        \begin{equation}\label{eq:RaV}
         \sup_{\boldsymbol y\in\mathbb R^d}   \int_{\mathbb R^d} R_V^a(\boldsymbol x,\boldsymbol y)\, d\boldsymbol x\leq C_a d,
        \end{equation}
which implies the corollary. \end{proof}

\begin{proof}[Proof of Corollary \ref{Coro2}]
    Set $\widetilde V(\boldsymbol x)=V_1(x_1)+V_2(x_2)+\dots +V_d(x_d)$, where, as above,  $\boldsymbol x=(x_1,x_2,\dots x_d)$. Let $k_t^{c\tilde V}(\boldsymbol x,\boldsymbol y)$, $k_t^{C\widetilde V}(\boldsymbol x,\boldsymbol y)$, $k_t^{ V}(\boldsymbol x,\boldsymbol y)$ denote the integral kernels of the semigrpups generated by $\Delta-c\widetilde V$, $\Delta -C\widetilde V$, $\Delta -V$, respectively. From the  Feynman-Kac formula, we conclude  that
    $$ k_t^{C\widetilde V}(\boldsymbol x,\boldsymbol y)\leq k_t^{V}(\boldsymbol x,\boldsymbol y)\leq k_t^{c\widetilde V}(\boldsymbol x,\boldsymbol y).$$
    Hence, see \eqref{kernel-R_a},
    $$ 0\leq R_V^a(\boldsymbol x,\boldsymbol y)\leq \Big(\frac{C}{c}\Big)^a R_{c\widetilde V}^a(\boldsymbol x,\boldsymbol y).$$
    Now \eqref{eq:VV} follows from \eqref{eq:RaV} applied to $c\widetilde V$.
\end{proof}


\begin{thebibliography}{}
\bibitem{CZ} W. Czaja, J. Zienkiewicz,  \emph{Atomic characterization of the Hardy space  $H^1_L(\mathbb R)$   of one-dimensional Schrödinger operators with nonnegative potentials,}
Proc. Amer. Math. Soc. 136 (2008), no. 1, 89--94.


\bibitem{Davies} E.B. Davies, \emph{One parameter semi-groups,}
(Academic Press, London, 1980), viii + 230 pp.


\bibitem{K-W} M. Kucharski and B. Wróbel, \emph{On
$L^p$
 estimates for positivity-preserving Riesz transforms related to Schrödinger operators}, Annales de l’Institut Fourier,  to appear, DOI : 10.5802/aif.3744.

 \bibitem{K2} M. Kucharski, \emph{Dimension-free estimates for positivity-preserving Riesz transforms related to Schrödinger operators with certain potentials}, Studia Math. 288 (2026), 285--299
\end{thebibliography}
\end{document}